\documentclass[a4paper,12pt]{article}

\usepackage[left=2cm,right=2cm, top=2cm,bottom=3cm,bindingoffset=0cm]{geometry}

\usepackage{verbatim}
\usepackage{amsmath}
\usepackage{amsthm}
\usepackage{amssymb}
\usepackage{delarray}
\usepackage{cite}
\usepackage{hyperref}
\usepackage{mathrsfs}
\usepackage{tikz}
\usetikzlibrary{patterns}
\usepackage{caption}
\DeclareCaptionLabelSeparator{dot}{. }
\newcommand{\al}{\alpha}
\newcommand{\be}{\beta}
\newcommand{\ga}{\gamma}
\newcommand{\de}{\delta}
\newcommand{\la}{\lambda}
\newcommand{\om}{\omega}

\newcommand{\eps}{\varepsilon}
\newcommand{\vv}{\varphi}

\theoremstyle{plain}

\numberwithin{equation}{section}

\newtheorem{thm}{Theorem}[section]
\newtheorem{lem}[thm]{Lemma}
\newtheorem{prop}[thm]{Proposition}
\newtheorem{cor}[thm]{Corollary}

\theoremstyle{definition}

\newtheorem{ip}[thm]{Inverse Problem}
\newtheorem{defin}[thm]{Definition}
\newtheorem{alg}[thm]{Algorithm}

\theoremstyle{remark}

\newtheorem{remark}[thm]{Remark}

\DeclareMathOperator*{\Res}{Res}
\DeclareMathOperator{\sign}{sign}

 \allowdisplaybreaks

\begin{document}

\begin{center}
{\Large\bf Quasiperiodic inverse Sturm-Liouville problem}
\\[0.5cm]
{\bf Natalia P. Bondarenko}$^{1,2}$\\[0.1cm]

\textit{email:} bondarenkonp@sgu.ru\\[0.2cm]

1. Department of Mechanics and Mathematics, Saratov State University, 
Astrakhanskaya 83, Saratov 410012, Russia, \\

2. S.M. Nikolskii Mathematical Institute, RUDN University, 6 Miklukho-Maklaya St, Moscow, 117198, Russia.
\end{center}

\vspace{0.5cm}

{\bf Abstract.} In this paper, the Sturm-Liouville problem with nonseparated quasiperiodic boundary conditions is considered. We study the recovery of the problem parameters from the Hill-type discriminant, the Dirichlet spectrum, and the sequence of signs. We obtain the necessary and sufficient conditions of solvability, the local solvability and stability, as well as the uniform stability for this inverse spectral problem.

\medskip


\medskip


\section{Introduction} \label{sec:intr}

Consider the boundary value problem $\mathcal L = \mathcal L(q, a, h)$ for the Sturm-Liouville equation
\begin{equation} \label{eqv}
-y'' + q(x) y = \rho^2 y, \quad x \in (0,1),
\end{equation}
with the quasiperiodic boundary conditions
\begin{equation} \label{bc}
y(0) = a y(1), \quad y'(0) - h y(0) = a^{-1} y'(1),
\end{equation}
where the so-called potential $q$ belongs to the class $L_{2,\mathbb R}(0,1)$ of real-valued square-integrable functions, $\rho^2$ is the spectral parameter, $a \in \mathbb R \setminus \{ -1, 0, 1 \}$, and $h \in \mathbb R$.

This paper is focused on the inverse spectral theory for the quasiperiodic problem \eqref{eqv}--\eqref{bc}.
We study the reconstruction of $q$, $a$, and $h$ from the Hill-type discriminant $d(\rho)$, the Dirichlet eigenvalues $\{ \rho_n^2 \}_1^{\infty}$, and the sequence of signs $\{ \sigma_n \}_1^{\infty}$ (see Section~\ref{sec:main} for details). In this paper, we for the first time obtain the necessary and sufficient conditions of solvability, the local solvability and stability, as well as the uniform stability for this inverse problem.

The majority of classical results for inverse Sturm-Liouville problems are concerned with separated (Dirichlet and Robin) boundary conditions (see, e.g., the monographs \cite{Mar77, Lev84, PT87, FY01} and references therein). Furthermore, the inverse spectral theory for the periodic ($a = 1, \, h = 0$) and antiperiodic ($a = -1, \, h = 0$) boundary conditions \eqref{bc} has been extensively investigated in \cite{Stan70, MO75, Yur81, TS97, Kor97, Kor99, Kor03, Yur17-tamk} and other studies. Note that the quasiperiodic boundary conditions in the case $a \in \mathbb R \setminus \{ -1, 0, 1 \}$ are strongly regular, so they fundamentally differ from the periodic and antiperiodic ones, which are regular but not strongly regular (see \cite[p.~268]{Fr12}). Therefore, in this paper, we exclude $a = \pm 1$ from the consideration. 

For the second-order differential equation \eqref{eqv}, there exist two irreducible types of self-adjoint nonseparated boundary conditions:
\begin{gather*} \label{type1}
y(0) = w y(1), \quad y'(0) - h y(0) = \overline{w}^{-1} y'(1), \tag{I} \\ \label{type2}
y'(0) + \be y(0) + w y(1) = 0, \quad y'(1) - \overline{w} y(0) + \ga y(1) = 0, \tag{II}
\end{gather*}
where $w \ne 0$ is complex and $h, \be, \ga$ are reals. Obviously, conditions \eqref{bc} are a special case of type~\eqref{type1}.

Inverse Sturm-Liouville problems with general nonseparated boundary conditions of types \eqref{type1} and \eqref{type2} started to be investigated by Plaksina~\cite{Plak78, Plak80, Plak88} and continued by Guseinov and Nabiev \cite{GN89, GN95} and by Gasymov et al \cite{GGN90}. Later on, Guseinov and Nabiev generalized their methods and results to quadratic differential pencils \cite{GN07} and to the case of eigenparameter-dependent nonseparated boundary conditions \cite{Nab22} (see also the other studies of these authors).
Plaksina followed the approach of Marchenko and Ostrovskii~\cite{MO75}, which allowed them to obtain the characterization of the periodic and antiperiodic eigenvalues. It is worth noting that the periodic and antiperiodic spectra are closely related to each other, being the zeros of the Hill discriminants $d(\rho) \pm 2$, and they do not uniquely determine the potential $q$.
Consequently, the early studies \cite{Plak78, Plak80, Plak88, GN89, GN95, GGN90} were mainly focused on establishing relations between pairs of spectra that corresponded to different sets of nonseparated boundary conditions. Another approach was suggested by Ralston and Trubowitz \cite{RT88}. They considered the Sturm-Liouville problems with the boundary conditions
$$
\begin{pmatrix}
a & b \\
c & d 
\end{pmatrix}
\begin{pmatrix}
y(1) \\ y'(1)
\end{pmatrix} =
\begin{pmatrix}
y(0) \\ y'(0)
\end{pmatrix},
$$
which generalize \eqref{bc}, and studied the isospectral sets of the problem parameters as analytic manifolds. 

The arguments in \cite{Plak88} and in \cite[Appendix~A]{RT88} actually imply the uniqueness of recovering the potential from the quasiperiodic spectrum generated by the boundary conditions \eqref{bc}, the Dirichlet spectrum, and the corresponding sequence of signs. The similar spectral data for the first time were introduced by Stankevich~\cite{Stan70} for the periodic problem. This type of spectral data appeared to be especially useful for applications to inverse spectral problems on quantum graphs (see \cite{Yur09, Yur14, Yur17-res, BS20, Bond25}). Another kind of inverse problems with nonseparated boundary conditions was investigated by Makin~\cite{Mak07, Mak08, Mak10}. He considered non-self-adjoint regular (but not strongly regular)  and irregular Sturm-Liouville problems and obtained the characteristic properties of their spectra, which do not uniquely specify the potential.

A significant progress in the study of inverse spectral problems with nonseparated boundary conditions has been achieved by Yurko \cite{Yur81, Yur00, Yur12-tamk, Yur12-amp, Yur16, Yur17-tamk} (see also the survey \cite{Yur20}). In particular, in \cite{Yur81,Yur00}, the spectral data characterization for the Sturm-Liouville problem with the nonseparated boundary conditions of type~\eqref{type2} was obtained. Moreover, local solvability and stability of the corresponding inverse problem were proved. 

The series of papers \cite{Yur12-tamk, Yur12-amp, Yur16} deals with the Sturm-Liouville operators and quadratic differential pencils induced by the general non-self-adjoint quasiperiodic boundary conditions
\begin{equation} \label{nsabc}
y(0) = \al y(T), \quad y'(0) - (i \rho h + h_0) y(0) = \be y'(T),
\end{equation}
where $\al$, $\be$, $h$, and $h_0$ can be arbitrary complex numbers. 
Yurko \cite{Yur12-tamk, Yur12-amp, Yur16} has proved the uniqueness theorems and obtained constructive procedures for solving the corresponding inverse spectral problems. We also mention that the investigation of inverse problems for quasiperiodic non-self-adjoint differential pencils was initiated by Buterin in the conference paper \cite{But09}. Recently, Zhang et al \cite{ZY22, ZWY24-res, ZWY24-aml} studied periodic inverse problems with discontinuities, which, in fact, are special cases of the quasiperiodic problems in \cite{Yur16, Yur20}. The boundary conditions \eqref{nsabc}, obviously, generalize \eqref{bc}. However, the results of the above-mentioned papers for the corresponding type of inverse spectral problems are limited to uniqueness and reconstruction. To the best of the author's knowledge, there were no results on solvability and stability of such inverse problems. So, this paper aims to fill this gap.

It is worth mentioning that
the Sturm-Liouville problems with discontinuity conditions
\begin{equation} \label{discont}
y(s + 0) = a y(s - 0), \quad y'(s + 0) = a^{-1} y'(s - 0) + c y(s - 0)
\end{equation}
inside an interval arise in various physical applications. Indeed, the Sturm-Liouville equation 
$$
-(p(x) y')' + q(x) y = \la w(x) y
$$ 
with a positive weight $w(x)$ having a jump at the point $s$ can be reduced to the Schr\"odinger form \eqref{eqv} with the transmission conditions \eqref{discont}. Inverse Sturm-Liouville problems with discontinuities appear in geophysics \cite{Hald84}, acoustics \cite{DP17}, and other fields in connection with discontinuous material properties (see \cite[Section~4.4]{FY01} for more references). Furthermore, the conditions \eqref{discont} describe $\de'$-interactions in quantum mechanics. A recent overview on this topic can be found in \cite{GH25}. The quasiperiodic boundary conditions \eqref{bc} correspond to a loop with a discontinuity of type \eqref{discont} at the point $x = 0$, which is connected to $x = 1$. Our motivation of studying the quasiperiodic inverse problem arises from the investigation of the above-mentioned quantum graphs with cycles.

In this paper, we obtain the spectral data characterization (Theorem~\ref{thm:nsc}) for the problem $\mathcal L$ in terms of the Hill-type discriminant $d(\rho)$, which can be constructed by using the eigenvalues of $\mathcal L$, the Dirichlet spectrum, and the corresponding sequence of signs. The proof method relies on the reconstruction algorithm from \cite{Yur12-tamk, Yur12-amp, Yur16}. Furthermore, we consider a small perturbation of the spectral data of the problem $\mathcal L$ and prove Theorem~\ref{thm:loc} on the local solvability and stability of the inverse problem. We point out that our proof method is simpler than the one from \cite{Yur00} for another type of nonseparated boundary conditions. We directly reduce the quasiperiodic inverse problem to the classical inverse Sturm-Liouville problem with the Dirichlet boundary conditions and apply the result from \cite{BSY13} on local solvability and stability of the latter problem. Moreover, we show the unconditional uniform stability of the quasiperiodic inverse problem (Theorem~\ref{thm:uni}). The results of this kind for the Sturm-Liouville problem with the Dirichlet boundary conditions have been obtained by Savchuk and Shkalikov \cite{SS10, SS13}. Specifically, they have found the sets of spectral data, on which the stability estimate is uniform. Our constructive approach allows us to transfer a result from \cite{SS10, SS13} to the quasiperiodic problem \eqref{eqv}--\eqref{bc}. The results of this paper can be applied to investigate solvability and stability of inverse spectral problems on quantum graphs (see \cite{Bond25}).

The paper is organized as follows. Section~\ref{sec:main} contains rigorous formulations of the quasiperiodic inverse problem and the main results. In Section~\ref{sec:class}, we provide the known results for the inverse Sturm-Liouville problem with the Dirichlet boundary conditions that are necessary for the proofs of the main theorems. In Section~\ref{sec:alg}, we adapt the reconstruction method from \cite{Yur12-tamk, Yur12-amp, Yur16} to the considered quasiperiodic problem. Section~\ref{sec:proof} contains the proof of the main results (Theorems~\ref{thm:nsc}, \ref{thm:loc}, and \ref{thm:uni}). In Section~\ref{sec:eigen}, we reformulate the main results by using the eigenvalues of the quasiperiodic problem \eqref{eqv}--\eqref{bc} instead of the Hill-type discriminant. 

\smallskip

Throughout this paper, we use the notation:
\begin{itemize}
\item The prime $y'$ denotes the derivative by $x$ or $t$, the dot $\dot y$, by $\la = \rho^2$.
\item $a_{\pm} := \frac{1}{2} (a \pm a^{-1})$.
\item $\sign a = \begin{cases}
1, & a > 0, \\
0, & a = 0, \\
-1, & a < 0.
\end{cases}$
\item $PW_{\pm}(0,1)$ are the classes of even and odd Paley-Wiener functions:
\begin{align*}
& PW_+(0,1) = \left\{ F_+(\rho) = \int_0^1 f(t) \cos \rho t \, dt \colon f \in L_{2,\mathbb R}(0,1) \right\}, \\
& PW_-(0,1) = \left\{ F_-(\rho) = \int_0^1 f(t) \sin \rho t \, dt \colon f \in L_{2,\mathbb R}(0,1) \right\}.
\end{align*}
\item $l_2^1$ is the linear normed space of the infinite sequences $\{ a_n \}_1^{\infty}$ such that 
$$
\| \{ a_n \}_1^{\infty} \|_{l_2^1} := \left( \sum\limits_{n = 1}^{\infty} (n |a_n|)^2 \right)^{1/2} < \infty.
$$
\item Along with $\mathcal L = \mathcal L(q, a, h)$, we consider the eigenvalue problem $\tilde{\mathcal L} = \mathcal L(\tilde q, a, h)$ of the same form \eqref{eqv}--\eqref{bc} but with a different potential $\tilde q \in L_{2,\mathbb R}(0,1)$. We agree that, if a symbol $\ga$ denotes an object related to $\mathcal L$, then the symbol $\tilde \ga$ with tilde will denote the similar object related to $\tilde{\mathcal L}$.
\end{itemize}

\section{Main results} \label{sec:main}

Denote by $\vv(x, \rho)$ and $S(x, \rho)$ the solutions of equation \eqref{eqv} under the initial conditions
$$
\vv(0, \rho) = 1, \quad \vv'(0, \rho) = h, \quad S(0,\rho) = 0, \quad S'(0,\rho) = 1.
$$

As the main spectral characteristics of the quasiperiodic problem $\mathcal L$, we use the function
\begin{equation} \label{defd}
d(\rho) := a \vv(1, \rho) + a^{-1} S'(1,\rho).
\end{equation}

We call $d(\rho)$ the Hill-type discriminant, because it generalizes the Hill discriminant $\vv(1,\rho) + S'(1,\rho)$, which arises in investigation of the periodic Sturm-Liouville problem (see \cite{Stan70}). Note that the eigenvalues of the problem $\mathcal L$ coincide with the squared zeros of the characteristic function 
\begin{equation} \label{defDelta}
\Delta(\rho) := d(\rho) - 2.
\end{equation}
Therefore, the Hill-type discriminant $d(\rho)$ can be easily recovered from the quasiperiodic eigenvalues (see Section~\ref{sec:eigen} for details).

Next, denote by $\{ \rho_n^2 \}_1^{\infty}$ ($-\tfrac{\pi}{2} < \arg \rho_n \le \tfrac{\pi}{2}$) the eigenvalues of the boundary value problem for equation \eqref{eqv} with the Dirichlet conditions $y(0) = y(1) = 0$. It is well-known that the eigenvalues $\{ \rho_n^2 \}_1^{\infty}$ are real and simple: $\rho_n^2 < \rho_{n+1}^2$, $n \ge 1$. 

Introduce the notations 
\begin{equation} \label{defH}
H(\rho) := a \vv(1,\rho) - a^{-1} S'(1,\rho), \quad \sigma_n := \sign H(\rho_n), \quad n \ge 1.
\end{equation}

We call the triple $\bigl( d(\rho), \{ \rho_n \}_1^{\infty}, \{ \sigma_n \}_1^{\infty} \bigr)$ the spectral data of the problem $\mathcal L$ and study the following inverse spectral problem:

\begin{ip} \label{ip:1}
Given $d(\rho)$, $\{ \rho_n \}_1^{\infty}$, and $\{ \sigma_n \}_1^{\infty}$, find $q$, $a$, and $h$.
\end{ip}

The results of \cite{Plak88} imply the uniqueness of solution for Inverse Problem~\ref{ip:1}. In this paper, we obtain the spectral data characterization:

\begin{thm} \label{thm:nsc}
Let $a \in \mathbb R \setminus \{ -1, 0, 1 \}$ be fixed. Then, for a collection $d(\rho)$, $\{ \rho_n \}_1^{\infty}$, $\{ \sigma_n \}_1^{\infty}$ to be the spectral data of a problem $\mathcal L$ of form \eqref{eqv}--\eqref{bc}, the following conditions are necessary and sufficient:
\begin{enumerate}
\item The function $d(\rho)$ has the form
\begin{equation} \label{intd}
d(\rho) = 2 a_+ \cos \rho + \frac{b \sin \rho}{\rho} + \frac{1}{\rho} \int_0^1 D(t) \sin \rho t \, dt, \quad b \in \mathbb R, \quad D \in L_{2,\mathbb R}(0,1).
\end{equation}
\item The numbers $\{ \rho_n^2 \}_1^{\infty}$ are real, $\rho_n^2 < \rho_{n+1}^2$, $n \ge 1$, and the asymptotic relation holds:
\begin{equation} \label{asymptrho}
\rho_n = \pi n + \frac{\om}{\pi n} + \frac{\varkappa_n}{n}, \quad \om \in \mathbb R, \quad \{ \varkappa_n \} \in l_2, \quad n \ge 1.
\end{equation}
\item $(-1)^n \sign a \cdot d(\rho_n) \ge 2$, $n \ge 1$.
\item $\sigma_n = 0$ iff $d(\rho_n) = \pm 2$.
\item $\sigma_n = (-1)^n \sign a_-$ for all sufficiently large $n$.
\end{enumerate}
\end{thm}

\begin{remark} \label{rem:omb}
By necessity, the constants $b$ and $\om$ from \eqref{intd} and \eqref{asymptrho}, respectively, are related to the parameters $(q, a, h)$ of the problem $\mathcal L$ as follows:
\begin{equation} \label{omb}
\om = \frac{1}{2} \int_0^1 q(x) \, dx, \quad b = a h + 2 a_+ \om.
\end{equation}
By sufficiency, $b$ and $\om$ can be any real numbers.
\end{remark}

Next, let us consider a small perturbation of spectral data, and formulate a local solvability and stability theorem. 

\begin{thm} \label{thm:loc}
Let $q \in L_{2,\mathbb R}(0,1)$, $a \in \mathbb R \setminus \{ -1, 0, 1 \}$, and $h \in \mathbb R$ be fixed, and let $\bigl( d(\rho), \{ \rho_n \}_1^{\infty}, \{ \sigma_n \}_1^{\infty} \bigr)$ be the spectral data of 
$\mathcal L(q, a, h)$. There exists $\eps > 0$ (which depends on $q$, $a$, $h$) such that, if a function $\tilde d(\rho)$ and reals $\{ \tilde \rho_n^2 \}_1^{\infty}$ satisfy the following conditions:
\begin{gather} \label{defdt}
\tilde d(\rho) = d(\rho) + \frac{1}{\rho} \int_0^1 \hat D(t) \sin \rho t \, dt, \quad \hat D \in L_{2, \mathbb R}(0, 1), \\ \label{ineqDrho}
\| \rho(d(\rho) - \tilde d(\rho)) \|_{L_2(\mathbb R)} \le \eps, \quad
\| \{ \rho_n - \tilde \rho_n \}_1^{\infty} \|_{l_2^1} \le \eps, \\ \label{condtd}
\tilde d(\tilde \rho_n) = \pm 2 \quad \text{iff} \quad \sigma_n = 0,
\end{gather}
then there exists a unique $\tilde q \in L_{2,\mathbb R}(0,1)$ such that $\bigl( \tilde d(\rho), \{ \tilde \rho_n \}_1^{\infty}, \{ \sigma_n \}_1^{\infty}\bigr)$ are the spectral data of the problem $\mathcal L(\tilde q, a, h)$. Moreover, the following estimate holds:
\begin{equation} \label{estq}
\| q - \tilde q \|_{L_2} \le C \left( \| \rho(d(\rho) - \tilde d(\rho)) \|_{L_2(\mathbb R)} + \| \{ \rho_n - \tilde \rho_n \}_1^{\infty} \|_{l_2^1} \right),
\end{equation}
where the constant $C > 0$ depends on $q$, $a$, and $h$. 
\end{thm}

Let us discuss the conditions of Theorem~\ref{thm:loc}.
According to \eqref{defdt}, the function $\rho(d(\rho) - \tilde d(\rho))$ belongs to $PW_-(0,1)$, so \eqref{ineqDrho} contains its norm in $L_2(\mathbb R)$. It follows from the asymptotics \eqref{intd} and \eqref{asymptrho} that the inequalities \eqref{ineqDrho} and the relations $\sigma_n = \tilde \sigma_n$ $(n \ge 1)$ together imply $a = \tilde a$ and $h = \tilde h$. Thus, Theorem~\ref{thm:loc} deals with small perturbations of the remainder terms in \eqref{intd} and \eqref{asymptrho}, and the difference of the potentials $\| q - \tilde q \|_{L_2}$ is estimated, while the coefficients $a$ and $h$ remain unchanged.
It is worth noting that there exists a wide class of spectral data perturbations that simultaneously satisfy the requirements \eqref{ineqDrho} and \eqref{condtd}, which is shown in Remark~\ref{rem:loc} below. 

Next, let us define a set of spectral data, on which the uniform stability of the inverse problem will be considered.

\begin{defin} \label{def:S}
Let $a$, $h$, $\om$, and $\{ \sigma_n \}_1^{\infty}$ be fixed so that $\sigma_n \in \{ -1, 0, 1 \}$ and condition~5 of Theorem~\ref{thm:nsc} is fulfilled. Let $\Omega, \de > 0$. Denote by $\mathcal S_{\Omega,\de} = \mathcal S_{\Omega, \de}\bigl(a, h, \om, \{ \sigma_n \}\bigr)$ the set of pairs $\bigl(d(\rho), \{ \rho_n \}_1^{\infty}\bigr)$ such that $d(\rho)$ has the form \eqref{intd} with $b = ah + 2a_+\om$ and $\| D \|_{L_2} \le \Omega$, the numbers $\rho_n$ are real and satisfy the asymptotics \eqref{asymptrho} with the remainders $\| \{ \varkappa_n \}_1^{\infty} \|_{l_2} \le \Omega$, and the following relations hold:
\begin{gather} \label{sep}
\rho_1 \ge 1, \quad \rho_{n+1} - \rho_n \ge \de, \quad n \ge 1, \\ \label{casede}
\begin{cases}
(-1)^n \sign a \cdot d(\rho_n) \ge 2 + \de & \text{if} \:\: \sigma_n \ne 0, \\ 
(-1)^n \sign a \cdot d(\rho_n) = 2  & \text{if} \:\: \sigma_n = 0.
\end{cases}
\end{gather}
\end{defin}

Note that, by virtue of Theorem~\ref{thm:nsc}, 
if $\left( d(\rho), \{ \rho_n \}_1^{\infty}\right) \in \mathcal S_{\Omega,\de}$, then
$\left( d(\rho), \{ \rho_n \}_1^{\infty}, \{ \sigma_n \}_1^{\infty}\right)$ are the spectral data of the problem $\mathcal L(q, a, h)$ for some $q \in L_{2,\mathbb R}(0,1)$ and $a$, $h$, $\{ \sigma_n \}_1^{\infty}$ being fixed by Definition~\ref{def:S}. The following theorem asserts the unconditional uniform stability of Inverse Problem~\ref{ip:1} on the set $\mathcal S_{\Omega, \de}$.

\begin{thm} \label{thm:uni}
Suppose that the values $a \in \mathbb R \setminus \{ -1, 0, 1 \}$, $h, \om \in \mathbb R$, $\sigma_n \in \{ -1, 0, 1 \}$ $(n \ge 1)$, $\Omega, \de > 0$ are fixed and $\sigma_n = (-1)^n \sign a_-$ for all sufficiently large $n$. Then, for any two pairs $\bigl( d(\rho), \{ \rho_n \}_1^{\infty}\bigr)$ and $\bigl( \tilde d(\rho), \{ \tilde \rho_n \}_1^{\infty}\bigr)$ in $\mathcal S_{\Omega, \de}$, the corresponding potentials $q$ and $\tilde q$ that solve Inverse Problem~\ref{ip:1} by the spectral data $\bigl( d(\rho), \{ \rho_n \}_1^{\infty}, \{ \sigma_n \}_1^{\infty}\bigr)$ and $\bigl( \tilde d(\rho), \{ \tilde \rho_n \}_1^{\infty}, \{ \sigma_n \}_1^{\infty}\bigr)$, respectively, satisfy the estimate \eqref{estq}, where the constant $C$ depends only on $a$, $h$, $\om$, $\{ \sigma_n \}_{n \ge 1}$, $\Omega$, and $\de$.
\end{thm}

\begin{remark}
The Hill-type discriminant for the Sturm-Liouville problem with the general boundary conditions of type~\eqref{type1} has the form $d(\rho) = w \vv(1,\rho) + \overline{w}^{-1} S'(1,\rho)$. It is different from \eqref{defd} for $a = |w|$ only by a constant multiplier $\dfrac{w}{|w|}$. Therefore, without loss of generality, we consider the boundary conditions \eqref{bc} but not \eqref{type1}.
\end{remark}

\section{Classical inverse Sturm-Liouville problem} \label{sec:class}

In this section, we provide the known results from \cite{BSY13, SS13} on the inverse Sturm-Liouville problem with the Dirichlet boundary conditions. These results will be used to prove the main theorems in the next sections.

Denote by $\mathcal L_0 = \mathcal L_0(q)$ the eigenvalue problem for equation \eqref{eqv} with the potential $q \in L_{2,\mathbb R}(0,1)$ and with the Dirichlet boundary conditions $y(0) = y(1) = 0$.
Denote by $C(x, \rho)$ the solution of equation \eqref{eqv} under the initial conditions $C(0,\rho) = 1$, $C'(0,\rho) = 0$.
The Weyl function of $\mathcal L_0$ is defined as $M(\rho^2) = \dfrac{C(1,\rho)}{S(1,\rho)}$. The function $M(\la)$ is meromorphic in $\la = \rho^2$ and has the simple poles at the eigenvalues $\{ \rho_n^2 \}_1^{\infty}$ of $\mathcal L_0$. Denote
\begin{equation} \label{defMn}
M_n := \Res_{\la = \rho_n^2} M(\la), \quad n \ge 1.
\end{equation}

We call $\{ M_n \}_1^{\infty}$ the Weyl sequence and
$\{ \rho_n, M_n \}_1^{\infty}$, the spectral data of the problem $\mathcal L_0$ or the Dirichlet spectral data.

Consider the classical inverse Sturm-Liouville problem:

\begin{ip} \label{ip:dir}
Given $\{ \rho_n, M_n \}_1^{\infty}$, find $q$.
\end{ip}

Note that $\{ S(x, \rho_n) \}_1^{\infty}$ are the eigenfunctions of $\mathcal L_0$ and 
$$
M_n = \al_n^{-1}, \quad \al_n := \int_0^1 S^2(x, \rho_n) \, dx, \quad n \ge 1.
$$

Thus, Inverse Problem~\ref{ip:dir} is equivalent to the recovery of the potential $q$ from the eigenvalues $\{ \rho_n^2 \}_1^{\infty}$ and the weight numbers $\{ \al_n \}_1^{\infty}$. For our technique, it will be more convenient to use the Weyl sequence $\{ M_n \}_1^{\infty}$, so in this section we formulate the results in terms of the spectral data $\{ \rho_n, M_n \}_1^{\infty}$.

The solution of Inverse Problem~\ref{ip:dir} is unique and can be found by standard methods of inverse spectral theory, in particular, by the Gelfand-Levitan method \cite{GL51, Lev84} and by the method of spectral mappings \cite{FY01}. The characteristic properties of the spectral data are described by the following proposition.

\begin{prop}[\hspace*{-3pt}\cite{BSY13}, Theorem 7] \label{prop:nsc} 
For numbers $\{ \rho_n, M_n \}_1^{\infty}$ to be the spectral data of a problem $\mathcal L_0$ with $q \in L_{2,\mathbb R}(0,1)$, the following conditions are necessary and sufficient:
$$
\rho_n^2 \in \mathbb R, \quad \rho_n^2 \ne \rho_m^2 \: (n \ne m), \quad M_n > 0, \quad n \ge 1,
$$
the values $\{ \rho_n \}$ fulfill the asymptotics \eqref{asymptrho} and
\begin{equation} \label{asymptM}
M_n = 2 (\pi n)^2 \left( 1 + \frac{\eta_n}{n} \right), \quad \{ \eta_n \} \in l_2, \quad n \ge 1.
\end{equation}
\end{prop}

The next proposition gives the local solvability and stability of Inverse Problem~\ref{ip:dir}.

\begin{prop}[\hspace*{-3pt}\cite{BSY13}, Theorem 8] \label{prop:loc}
Let $q \in L_{2,\mathbb R}(0,1)$ be fixed, and let $\{ \rho_n, M_n \}_1^{\infty}$ be the spectral data of $\mathcal L_0(q)$. There exists $\eps > 0$ (depending on $q$) such that, if real numbers $\{ \tilde \rho_n^2, \tilde M_n \}_1^{\infty}$ satisfy
$$
\| \{ \rho_n - \tilde \rho_n \}_1^{\infty} \|_{l_2^1} \le \eps, \quad
\| \{ n^{-2} (M_n - \tilde M_n) \}_1^{\infty} \|_{l_2^1} \le \eps,
$$
then there exists a unique $\tilde q \in L_{2,\mathbb R}(0,1)$, for which $\{ \tilde \rho_n, \tilde M_n \}_1^{\infty}$ are the spectral data of $\mathcal L_0(\tilde q)$. Moreover, the following estimate holds:
\begin{equation} \label{estqdir}
\| q - \tilde q \|_{L_2} \le C \left( \| \{ \rho_n - \tilde \rho_n \}_1^{\infty} \|_{l_2^1} + \| \{ n^{-2} (M_n - \tilde M_n) \}_1^{\infty} \|_{l_2^1} \right),
\end{equation}
where the constant $C$ depends only on $q$.
\end{prop}

For convenience, Propositions~\ref{prop:nsc} and~\ref{prop:loc} refer to the paper \cite{BSY13}, which is focused on the Dirichlet boundary conditions. However, these results can be deduced from earlier literature. In particular, Marchenko and Ostrovskii \cite{MO75} for the first time have obtained the criterion for two spectra of the Sturm-Liouville equation \eqref{eqv} with real-valued square-integrable potential (see also \cite[Theorem~3.4.1]{Mar77}). That criterion can be used to get Proposition~\ref{prop:nsc}. 
In the monograph by Freiling and Yurko \cite{FY01}, the results similar to Propositions~\ref{prop:nsc} and~\ref{prop:loc} are discussed in detail for the Robin boundary conditions. P\"oschel and Trubowitz \cite{PT87} studied the inverse Sturm-Liouville problem using spectral data $\{ \rho_n^2, \log S'(1,\rho_n) \}_1^{\infty}$, which can be uniquely determined by $\{ \rho_n, M_n \}_1^{\infty}$ and vice versa. The real-analytic isomorphism that was constructed in \cite{PT87} implies a local theorem analogous to Proposition~\ref{prop:loc}.

The uniform stability of the inverse Sturm-Liouville problems has been proved by Savchuk and Shkalikov \cite{SS10, SS13} for potentials in the scale of the Sobolev spaces $W_2^{\al}[0,1]$, $\al > -1$. For convenience, we reformulate their result in terms of the spectral data $\{ \rho_n, M_n \}_1^{\infty}$ for $q \in L_2(0,1)$ and a fixed $\omega$.

\begin{defin} \label{def:S0}
Fix $\om \in \mathbb R$.
For $\Omega > 0$ and $\de > 0$, introduce the set $\mathcal S_{\Omega,\de}^0 = S_{\Omega,\de}^0(\om)$ of the real sequences $\{ \rho_n, M_n \}_1^{\infty}$ satisfying the asymptotics \eqref{asymptrho} and \eqref{asymptM} with the fixed $\om$, the estimates $\| \{ \varkappa_n \}_1^{\infty} \|_{l_2} \le \Omega$ and $\| \{ \eta_n \}_1^{\infty} \|_{l_2} \le \Omega$ for the remainders, the inequalities \eqref{sep} and $M_n \ge \de$, $n \ge 1$.
\end{defin}

According to Proposition~\ref{prop:nsc} and Definition~\ref{def:S0}, every sequence $\{ \rho_n, M_n \}_1^{\infty}$ of $\mathcal S_{\Omega, \de}^0$ is the spectral data of some problem $\mathcal L_0$.

\begin{prop}[\hspace*{-3pt}\cite{SS13}, Theorem 3.15] \label{prop:uni}
Let $\om \in \mathbb R$, $\Omega > 0$, and $\de > 0$ be fixed. Then, for any $\{ \rho_n, M_n \}_1^{\infty}$ and $\{ \tilde \rho_n, \tilde M_n \}_1^{\infty}$ in $\mathcal S_{\Omega, \de}^0$, the respective solutions $q$ and $\tilde q$ of Inverse Problem~\ref{ip:dir} satisfy the estimate \eqref{estqdir}, where the constant $C$ depends only on $\omega$, $\Omega$, and $\de$.
\end{prop}

\section{Reconstruction algorithm} \label{sec:alg}

In this section, we describe the constructive solution of Inverse Problem~\ref{ip:1} that has been obtained in \cite{Yur12-tamk, Yur12-amp, Yur16} for quasiperiodic problems of more general form than \eqref{eqv}--\eqref{bc}. The main idea consists in the reduction of Inverse Problem~\ref{ip:1} to Inverse Problem~\ref{ip:dir}.

Using \eqref{defd} and \eqref{defH}, we get
\begin{align} \label{relvv}
& \vv(1,\rho) = \frac{1}{2a} (d(\rho) + H(\rho)), \\ \label{relSp}
& S'(1,\rho) = \frac{a}{2} (d(\rho) - H(\rho)).
\end{align}

Note that
\begin{equation} \label{wron}
\vv(x, \rho) S'(x, \rho) - \vv'(x, \rho) S(x, \rho) \equiv 1.
\end{equation}

Substituting \eqref{relvv} and \eqref{relSp} into \eqref{wron} for $x = 1$ and using the relation $S(1,\rho_n) = 0$ for the Dirichlet eigenvalues $\{ \rho_n^2 \}_1^{\infty}$, we obtain
\begin{equation} \label{Hd2}
H^2(\rho_n) = d^2(\rho_n) - 4, \quad n \ge 1.
\end{equation}
Recall that $\sigma_n = \sign H(\rho_n)$, so
\begin{equation} \label{findH}
H(\rho_n) = \sigma_n \sqrt{d^2(\rho_n) - 4}, \quad n \ge 1.
\end{equation}

Relations \eqref{relvv} and \eqref{findH} together yield
\begin{equation} \label{findvv}
\vv(1,\rho_n) = \frac{1}{2a} \left( d(\rho_n) + \sigma_n \sqrt{d^2(\rho_n) - 4} \right), \quad n \ge 1.
\end{equation}

Note that $\vv(1, \rho) = C(1,\rho) + h S(1,\rho)$.
Consequently, the definition \eqref{defMn} implies
\begin{equation} \label{findMn}
M_n = \Res_{\la = \rho_n^2} \frac{C(1,\sqrt{\la})}{S(1,\sqrt \la)} = \Res_{\la = \rho_n^2} \frac{\vv(1,\sqrt{\la})}{S(1,\sqrt \la)} = \frac{\vv(1, \rho_n)}{\dot r(\rho_n^2)}, 
\end{equation}
where $r(\rho^2) = S(1,\rho)$ and $\dot r(\la) = \frac{d}{d\la} r(\la)$, $\la = \rho^2$.

The function $r(\la)$ can be found as the infinite product by using its zeros (see, e.g., \cite[Section 1.1]{FY01}):
\begin{equation} \label{prodr}
r(\la) = \pi\prod_{n = 1}^{\infty} \frac{\rho_n^2 - \la}{n^2}.
\end{equation}

Thus, given $d(\rho)$, $\{ \rho_n \}_1^{\infty}$, and $\{ \sigma_n \}_1^{\infty}$, one can construct the Weyl sequence $\{ M_n \}_{n \ge 1}$ by formulas \eqref{findMn}, \eqref{findvv}, and \eqref{prodr}. Afterwards, the potential $q$ can be recovered by solving Inverse Problem~\ref{ip:dir}.

Let us discuss the reconstruction of the parameters $a$ and $h$. For this purpose, we need some asymptotical analysis. The standard technique (see, e.g., \cite[Section~1.1]{FY01}) implies the following representations:
\begin{align} \label{asymptvv}
& \vv(1, \rho) = \cos \rho + (h + \om) \frac{\sin \rho}{\rho} +  \frac{\ga_1(\rho)}{\rho}, \\ \label{asymptSp}
& S'(1,\rho) = \cos \rho + \frac{\om \sin \rho}{\rho} + \frac{\ga_2(\rho)}{\rho},
\end{align}
where $\om$ is defined by \eqref{omb} and $\ga_j(\rho) \in PW_-(0,1)$, $j = 1, 2$.

Substituting \eqref{asymptvv} and \eqref{asymptSp} into \eqref{defd}, we arrive at \eqref{intd}, where $b$ is given by \eqref{omb}. It follows from \eqref{intd} that
\begin{align} \label{finda+}
& a_+ = \frac{1}{2} \lim_{n \to \infty} d(2\pi n), \\ \label{findb}
& b = \lim_{n \to \infty} \bigl( 2 \pi n + \tfrac{\pi}{2}\bigr) d\bigl( 2 \pi n + \tfrac{\pi}{2}\bigr),
\end{align}
where $n$ runs over natural numbers.

Using \eqref{intd} and \eqref{asymptrho}, we obtain
\begin{equation} \label{asymptdn}
d(\rho_n) = 2 a_+ (-1)^n \left( 1 + \frac{\kappa_n}{n} \right), \quad \{ \kappa_n \} \in l_2, \quad n \ge 1.
\end{equation}
Analogously, substituting \eqref{asymptvv} and \eqref{asymptSp} into \eqref{defH} and using \eqref{asymptrho}, we get
\begin{equation} \label{asymptHn}
H(\rho_n) = 2 a_- (-1)^n \left( 1 + \frac{\tau_n}{n} \right), \quad \{ \tau_n \} \in l_2, \quad n \ge 1.
\end{equation}

Using \eqref{findH} and comparing \eqref{asymptdn} with \eqref{asymptHn}, we conclude that $\sigma_n = (-1)^n \sign a_-$ for all sufficiently large $n$. From the latter relation, the sign of $a_-$ can be found:
\begin{equation} \label{findsigna-}
\sign a_- = \lim_{n \to \infty} (-1)^n \sigma_n.
\end{equation}
Therefore, one can recover
\begin{equation} \label{finda}
a = a_+ + \sign a_- \sqrt{a_+^2 - 1}.
\end{equation}

Next, the constants $\om$ and $h$ are easily found from \eqref{asymptrho} and \eqref{omb}, respectively:
\begin{equation} \label{findomh}
\om = \lim_{n \to \infty} \pi n (\rho_n - \pi n), \quad
h = a^{-1} (b - 2 a_+ \om).
\end{equation}

Thus, we arrive at the following algorithm for solving Inverse Problem~\ref{ip:1}.

\begin{alg} \label{alg:1}
Suppose that the spectral data $d(\rho)$, $\{ \rho_n \}_{n \ge 1}$, and $\{ \sigma_n \}_{n \ge 1}$ are given. We have to find $q$, $a$, and $h$.
\begin{enumerate}
\item Find $a$ using \eqref{finda+}, \eqref{findsigna-}, and \eqref{finda}.
\item Find $h$ using \eqref{findb} and \eqref{findomh}.
\item Calculate the values $\vv(1,\rho_n)$ ($n \ge 1$) by formula \eqref{findvv}.
\item Construct the function $r(\la)$ as the infinite product \eqref{prodr} and find the derivative $\dot r(\la)$.
\item Find the numbers $\{ M_n \}_{n \ge 1}$ by \eqref{findMn}.
\item Recover the potential $q$ from the Dirichlet spectral data $\{ \rho_n, M_n \}_{n \ge 1}$ by any standard method (see \cite{Lev84, FY01}).
\end{enumerate}
\end{alg}

\section{Proofs} \label{sec:proof}

In this section, we prove Theorems~\ref{thm:nsc}, \ref{thm:loc}, and~\ref{thm:uni}. Their proofs are based on the reduction to Propositions~\ref{prop:nsc}, \ref{prop:loc}, and \ref{prop:uni}, respectively, for the classical inverse problem with the Dirichlet boundary conditions. For this purpose, we use Algorithm~\ref{alg:1} for constructive solution of the quasiperiodic inverse problem.

\begin{proof}[Proof of Theorem~\ref{thm:nsc}]
\textit{Necessity.} Suppose that $d(\rho)$, $\{ \rho_n \}_1^{\infty}$, and $\{ \sigma_n \}_1^{\infty}$ are the spectral data of the problem $\mathcal L$. Let us show that conditions 1--5 hold.

Condition 1 follows from \eqref{defd}, \eqref{asymptvv}, and \eqref{asymptSp}.

Condition 2 is valid due to Proposition~\ref{prop:nsc}.

Relation \eqref{Hd2} implies
\begin{equation} \label{d24}
d^2(\rho_n) \ge 4, \quad n \ge 1.
\end{equation}
Consequently, independently of $\sign H(\rho_n)$, we have
$$
\sign \left( d(\rho_n) + \sigma_n \sqrt{d^2(\rho_n) - 4}\right) = \sign d(\rho_n).
$$
Then, it follows from \eqref{findvv} that 
$$
\sign d(\rho_n) = \sign a \cdot \sign \vv(1, \rho_n).
$$
For the infinite product \eqref{prodr}, we have 
$$
\sign \dot r(\rho_n^2) = (-1)^n. 
$$
Using \eqref{findMn}, we get
\begin{equation} \label{signMn}
\sign M_n = (-1)^n \sign a \cdot \sign d(\rho_n).
\end{equation}
By Proposition~\ref{prop:nsc}, there holds $\sign M_n = 1$, so
\eqref{signMn} together with \eqref{d24} prove condition~3 of Theorem~\ref{thm:nsc}.

Condition 4 readily follows from \eqref{Hd2}.

Condition 5 has already been proved in Section~\ref{sec:alg}. This condition follows from \eqref{findH}, \eqref{asymptdn}, and \eqref{asymptHn}.

\medskip

\textit{Sufficiency}. Suppose that $a \in \mathbb R \setminus \{ -1, 0, 1 \}$ and data $d(\rho)$, $\{ \rho_n \}_1^{\infty}$, $\{ \sigma_n \}_1^{\infty}$ fulfill conditions 1--5 of Theorem~\ref{thm:nsc}. Let us construct the potential $q$ following steps 2--6 of Algorithm~\ref{alg:1}. We have to show that all the steps are implemented correctly. For steps 2--5, this is obvious. Indeed, the expression under the square root at step 3 is nonnegative due to condition 3 and the infinite product \eqref{prodr} converges to an entire analytic function in view of the asymptotics \eqref{asymptrho}. 

For the correctness of step~6, we have to prove that the values $\{ \rho_n, M_n \}_1^{\infty}$ fulfill the conditions of Proposition~\ref{prop:nsc}. For $\{ \rho_n \}_1^{\infty}$, this is trivial due to condition~2 of Theorem~\ref{thm:nsc}. Similarly to the proof of the necessity part, we get relation \eqref{signMn} for $\sign M_n$, which together with condition~3 implies $M_n > 0$. The asymptotics \eqref{asymptdn} is also valid, since it follows from \eqref{intd} and \eqref{asymptrho}. Using \eqref{asymptdn}, \eqref{findvv}, and condition~5, we obtain 
\begin{equation} \label{asymptvvn}
\vv(1, \rho_n) = (-1)^n \left( 1 + \frac{s_n}{n} \right), \quad \{ s_n \} \in l_2, \quad n \ge 1.
\end{equation}

Relations \eqref{prodr} and \eqref{asymptrho} together imply
\begin{equation} \label{intr}
r(\rho^2) = \frac{\sin \rho}{\rho} + \frac{\om \cos \rho}{\rho^2} + \frac{\ga_3(\rho)}{\rho^2}, \quad \ga_3 \in PW_+(0,1).
\end{equation}
Differentiation with respect to $\la = \rho^2$ yields
\begin{equation} \label{relrn}
\dot r(\rho^2) = \frac{1}{2\rho} \frac{d r(\rho^2)}{d \rho} = \frac{1}{2 \rho^2} \left( \cos \rho - (\om+1)\frac{\sin \rho}{\rho} + \frac{\ga_4(\rho)}{\rho}\right), \quad \ga_4 \in PW_-(0,1).
\end{equation}
Hence
\begin{equation} \label{asymptrn}
\dot r(\rho_n^2) = \frac{(-1)^n}{2 (\pi n)^2} \left( 1 + \frac{g_n}{n} \right), \quad \{ g_n \} \in l_2, \quad n \ge 1.
\end{equation}

Substituting \eqref{asymptvvn} and \eqref{asymptrn} into \eqref{findMn}, we arrive at the asymptotics \eqref{asymptM}. Thus, the Weyl sequence $\{ M_n \}_1^{\infty}$ satisfies all the conditions of Proposition~\ref{prop:nsc}. Therefore, $\{ \rho_n, M_n \}_1^{\infty}$ are the Dirichlet spectral data of some potential $q \in L_{2,\mathbb R}(0,1)$.

It remains to show that $d(\rho)$ and $\{ \sigma_n \}_1^{\infty}$ are spectral data of the quasiperiodic Sturm-Liouville problem \eqref{eqv}--\eqref{bc} for the reconstructed potential $q$. Suppose that, on the contrary, the problem $\mathcal L(q, a, h)$ has another Hill-type discriminant $\tilde d(\rho)$ and a sequence of signs $\{ \tilde \sigma_n \}_1^{\infty}$. Then, it follows from \eqref{findvv} that
\begin{equation} \label{dtd}
d(\rho_n) + \sigma_n \sqrt{d^2(\rho_n) - 4} = \tilde d(\rho_n) + \tilde \sigma_n \sqrt{\tilde d^2(\rho_n) - 4}, \quad n \ge 1.
\end{equation}

Note that the both values $d(\rho_n)$ and $\tilde d(\rho_n)$ satisfy condition~3 of Theorem~\ref{thm:nsc}. Consequently, considering the nine possible cases $\sigma_n, \tilde \sigma_n \in \{ -1, 0, 1 \}$, we conclude that \eqref{dtd} implies $d(\rho_n) = \tilde d(\rho_n)$ and $\sigma_n = \tilde \sigma_n$ for all $n \ge 1$. Furthermore, the functions $d(\rho)$ and $\tilde d(\rho)$ admit the representations \eqref{intd} and
\begin{equation} \label{intdt}
\tilde d(\rho) = 2 a_+ \cos \rho + \frac{b \sin \rho}{\rho} + \frac{1}{\rho} \int_0^1 \tilde D(t) \sin \rho t \, dt, \quad 
\tilde D \in L_{2,\mathbb R}(0,1),
\end{equation}
respectively.
Hence
\begin{equation} \label{intDtD}
\int_0^1 (D(t) - \tilde D(t)) \sin \rho_n t \, dt = 0, \quad n \ge 1.
\end{equation}

Since the values $\{ \rho_n \}_1^{\infty}$ are all distinct and satisfy the asymptotics \eqref{asymptrho}, then the sequence $\{ \sin \rho_n t \}_1^{\infty}$ is complete in $L_2(0,1)$ (see \cite{HV01}). Therefore, it follows from \eqref{intDtD} that $D(t) = \tilde D(t)$ a.e. on $(0,1)$. Thus $d(\rho) \equiv \tilde d(\rho)$, which concludes the proof.
\end{proof}

Proceed to proving Theorem~\ref{thm:loc}. Without loss of generality, let us assume that $\{ \rho_n \}_1^{\infty}$ are real and $\rho_1 \ge 1$. This can be achieved by a shift of the spectral parameter $\rho^2 \mapsto \rho^2 + c$, where $c$ is a constant.

\begin{lem} \label{lem:difdn}
Let $d(\rho)$ and $\{ \rho_n \}_1^{\infty}$ satisfy the hypothesis of Theorem~\ref{thm:loc}. Then, there exists $\eps > 0$ such that, for any $\tilde d(\rho)$ and any reals $\{ \tilde \rho_n \}_1^{\infty}$ satisfying the conditions \eqref{defdt} and \eqref{ineqDrho}, there holds
\begin{equation} \label{difdn}
|d(\rho_n) - \tilde d(\tilde \rho_n)| \le C \left( |\rho_n - \tilde \rho_n| + \frac{1}{n} |\hat D_n| + \frac{1}{n^2} \| \hat D \|_{L_2}\right), \quad n \ge 1,
\end{equation}
where the constant $C$ depends on $d(\rho)$ and $\{ \rho_n \}_1^{\infty}$, the function $\hat D$ was introduced in \eqref{defdt},
and
$$
\hat D_n := \int_0^1 \hat D(t) \sin nt \, dt, \quad n \ge 1.
$$
\end{lem}

\begin{proof}
Clearly, there holds
\begin{equation} \label{est1}
|d(\rho_n) - \tilde d(\tilde \rho_n)| \le |d(\rho_n) - \tilde d(\rho_n)| + |\tilde d(\rho_n) - \tilde d(\tilde \rho_n)|.
\end{equation}

Using \eqref{intd},  \eqref{asymptrho}, and \eqref{defdt}, we obtain
\begin{align*}
\tilde d(\rho_n) - d(\rho_n) & = \frac{1}{\rho_n} \int_0^1 \hat D(t) \sin \rho_n t \, dt \\
& = \frac{1}{n} \int_0^1 \hat D(t) \sin n t \, dt + \int_0^1 \hat D(t) O(n^{-2}) dt,
\end{align*}
where the $O$-estimate is uniform by $t \in [0,1]$. Hence
\begin{equation} \label{est2}
|d(\rho_n) - \tilde d(\rho_n)| \le C \left( \frac{1}{n} |\hat D_n| + \frac{1}{n^2} \| \hat D \|_{L_2} \right).
\end{equation}

By Lagrange's mean value theorem, we get
\begin{equation} \label{est3}
|\tilde d(\rho_n) - \tilde d(\tilde \rho_n)| \le \left| \tfrac{d}{d\rho} \tilde d(\xi_n)\right| |\rho_n - \tilde \rho_n|,
\end{equation}
where $\xi_n$ is located between $\rho_n$ and $\tilde \rho_n$. In view of \eqref{intdt}, the derivative $\tfrac{d}{d\rho} \tilde d(\rho)$ is bounded for real $\rho$ by a constant $C$ depending on $\hat D$. If $\| \rho(d(\rho) - \tilde d(\rho)) \|_{L_2(\mathbb R)}$ is sufficiently small, then
one can choose $C$ depending only on $d(\rho)$ and $\{ \rho_n \}_1^{\infty}$. Thus, combining \eqref{est1}, \eqref{est2}, and \eqref{est3}, we arrive at \eqref{difdn}.
\end{proof}

\begin{proof}[Proof of Theorem~\ref{thm:loc}]
Let $d(\rho)$, $\{ \rho_n \}_1^{\infty}$, and $\{ \sigma_n \}_1^{\infty}$ be the spectral data of a fixed problem $\mathcal L(q, a, h)$. Then, it follows from conditions 2, 3, and 4 of Theorem~\ref{thm:nsc} and the asymptotics \eqref{asymptdn} that relations \eqref{sep} and \eqref{casede} hold with some $\de > 0$, which depends on $q$, $a$, and $h$.

Let us show that, if $\tilde d(\rho)$ and $\{ \tilde \rho_n \}_1^{\infty}$ fulfill the hypothesis of Theorem~\ref{thm:loc} with sufficiently small $\eps > 0$, then the data $\tilde d(\rho)$, $\{ \tilde \rho_n \}_1^{\infty}$, $\{ \sigma_n \}_1^{\infty}$ satisfy the conditions of Theorem~\ref{thm:nsc}. Indeed, conditions 1, 4, and 5 are satisfied trivially. The inequality $\| \{ \rho_n - \tilde \rho_n \}_1^{\infty} \|_{l_2^1} \le \eps$ for sufficiently small $\eps > 0$ together with \eqref{sep} imply condition~2 for $\{ \tilde \rho_n \}_1^{\infty}$. It follows from \eqref{defdt} and Plancherel's Theorem that 
\begin{equation} \label{plan}
\| \rho (d(\rho) - \tilde d(\rho)) \|_{L_2(\mathbb R)} = \sqrt \pi \| \hat D \|_{L_2(0, 1)}. 
\end{equation}
So, the conditions \eqref{ineqDrho} and Lemma~\ref{lem:difdn} yield $|d(\rho_n) - \tilde d(\tilde \rho_n)| \le C \eps$ for all $n \ge 1$. 
Here and below in this proof, the constant $C$ depends only on $q$, $a$, and $h$.
Consequently, taking the relations
\eqref{condtd} and \eqref{casede} into account, we get
\begin{equation} \label{casedt}
\begin{cases}
(-1)^n \sign a \cdot \tilde d(\tilde \rho_n) \ge 2 + \tfrac{\de}{2} & \text{if} \: \sigma_n \ne 0, \\
(-1)^n \sign a \cdot \tilde d(\tilde \rho_n) = 2 & \text{if} \: \sigma_n = 0,
\end{cases}
\end{equation}
when $\eps > 0$ is chosen sufficiently small. This implies condition~3 of Theorem~\ref{thm:nsc}. 

Applying Theorem~\ref{thm:nsc}, we conclude that there exists a potential $\tilde q \in L_{2, \mathbb R}(0,1)$ (which is obviously unique) such that $\tilde d(\rho)$, $\{ \tilde \rho_n \}_1^{\infty}$, and $\{ \sigma_n \}_1^{\infty}$ are the spectral data of the problem $\mathcal L(\tilde q, a, h)$. Thus, we have proved the local solvability of Inverse Problem~\ref{ip:1}.

It remains to obtain the stability estimate \eqref{estq}. Consider the values $H(\rho_n)$ given by \eqref{findH} and
$$
\tilde H(\tilde \rho_n) = \sigma_n \sqrt{\tilde d^2(\tilde \rho_n) - 4}, \quad n \ge 1.
$$
If $\sigma_n = 0$, then $H(\rho_n) = \tilde H(\tilde \rho_n)$. In the case $\sigma_n \ne 0$, we have
\begin{equation} \label{difHn}
H(\rho_n) - \tilde H(\tilde \rho_n) = \sigma_n \frac{(d(\rho_n) - \tilde d(\tilde \rho_n))(d(\rho_n) + \tilde d(\tilde \rho_n))}{\sqrt{d^2(\rho_n) - 4} + \sqrt{\tilde d^2(\tilde \rho_n) - 4}}.
\end{equation}

According to \eqref{casede} and \eqref{casedt}, the denominator in \eqref{difHn} is bounded from below by a positive constant. In view of \eqref{intd}, \eqref{asymptrho}, \eqref{ineqDrho}, and \eqref{intdt}, the absolute value of $d(\rho_n) + \tilde d(\tilde \rho_n)$ is bounded from above. Hence, relation \eqref{difHn} implies
\begin{equation} \label{estHn}
|H(\rho_n) - \tilde H(\tilde \rho_n)| \le C |d(\rho_n) - \tilde d(\tilde \rho_n)|, \quad n \ge 1.
\end{equation}

Using \eqref{relvv} and \eqref{estHn}, we get
\begin{equation} \label{difvv}
|\vv(1,\rho_n) - \tilde \vv(1, \tilde \rho_n)| \le C |d(\rho_n) - \tilde d(\tilde \rho_n)|, \quad n \ge 1.
\end{equation}

Next, we construct the function $\tilde r(\la)$ similarly to \eqref{prodr} and, by using the representation \eqref{relrn}, derive the estimate
\begin{equation} \label{difrn}
|\dot r(\rho_n^2) - \dot{\tilde r}(\tilde \rho_n^2)| \le \frac{C}{n^2} \left( |\rho_n - \tilde \rho_n| + \frac{1}{n^2} \| \{ \rho_k - \tilde \rho_k \}_1^{\infty} \|_{l_2^1}\right), \quad n \ge 1.
\end{equation}

Combining \eqref{findMn}, \eqref{difvv}, \eqref{difrn} together with \eqref{asymptvvn}, \eqref{asymptrn} and the analogous relations for $\tilde \vv(1, \tilde \rho_n)$ and $\dot{\tilde r}(\tilde \rho_n^2)$, we arrive at the estimate
$$
|M_n - \tilde M_n| \le C n^2 \left( |\rho_n - \tilde \rho_n| + |d(\rho_n) - \tilde d(\tilde \rho_n)| + \frac{1}{n^2} \| \{ \rho_k - \tilde \rho_k \}_1^{\infty} \|_{l_2^1} \right).
$$
Taking \eqref{difdn} and \eqref{plan} into account, we deduce
\begin{equation} \label{difMn}
\| \{ n^{-2} (M_n - \tilde M_n)\}_1^{\infty} \}_{l_2^1} \| \le C\bigl( \| \rho(d(\rho) - \tilde d(\rho)) \|_{L_2(\mathbb R)} +  \| \{ \rho_n - \tilde \rho_n \}_1^{\infty} \|_{l_2^1} \bigr).
\end{equation}
Therefore, for sufficiently small $\eps > 0$, the conclusion of Proposition~\ref{thm:loc} holds. Thus, estimates \eqref{estqdir} and \eqref{difMn} together imply \eqref{estq}.
\end{proof}

\begin{remark} \label{rem:loc}
Let us show that the class of spectral data perturbations that simultaneously satisfy \eqref{ineqDrho} and \eqref{condtd} is sufficiently wide. In view of condition 5 of Theorem~\ref{thm:nsc}, there is a finite number of indices $n$ such that $\sigma_n = 0$. Suppose that $\tilde \rho_n = \rho_n$ for all such $n$. Then the requirement \eqref{condtd} together with \eqref{intd} and \eqref{intdt} imply \eqref{intDtD} for this finite set of indices $n$. In other words, the corresponding coordinates of $\tilde D(t)$ by the biorthonormal system to the Riesz basis $\{ \sin \rho_n t \}_1^{\infty}$ (see \cite{HV01}) coincide with the ones of $D(t)$. The other coordinates of $\tilde D$, as well as the values $\rho_n$ for $\sigma_n \ne 0$, can be chosen with a high degree of arbitrariness to satisfy \eqref{ineqDrho}.
\end{remark}

\begin{proof}[Proof of Theorem~\ref{thm:uni}]
Suppose that the values $a$, $h$, $\om$, $\{ \sigma_n \}_1^{\infty}$, $\Omega$, and $\de$ are fixed according to the hypothesis of the theorem. Let $\left( d(\rho), \{ \rho_n \}_1^{\infty} \right)$ belong to the set $\mathcal S_{\Omega,\de}$, which was introduced in Definition~\ref{def:S}. Then, one can construct the Weyl sequence $\{ M_n \}_1^{\infty}$ by following steps 3--5 of Algorithm~\ref{alg:1}. 
Let us show that the data $\{ \rho_n, M_n \}_1^{\infty}$ belong to the set $\mathcal S_{\Omega_0,\de_0}^0$ described in Definition~\ref{def:S0}, where $\Omega_0 > 0$ and $\de_0 > 0$ are some constants depending on $\Omega$ and $\de$. Specifically, we have to prove the estimates
$M_n \ge \de_0$ for all $n \ge 1$ and $\| \{ \eta_n \}_1^{\infty} \|_{l_2} \le \Omega_0$, where $\{ \eta_n \}$ are the remainder terms from the asymptotics \eqref{asymptM}.

Throughout this proof, the symbols $C$ and $c$ denote various positive constants depending only on $a$, $h$, $\om$, $\{ \sigma_n \}_1^{\infty}$, $\Omega$, and $\de$. It follows from \eqref{intd}, \eqref{asymptrho} and the bounds $\| \{ \varkappa_n \}_1^{\infty} \|_{l_2} \le \Omega$, $\| D \|_{L_2} \le \Omega$ from Definition~\ref{def:S} that $|d(\rho_n)| \le C$, $n \ge 1$. In view of \eqref{findvv} and \eqref{casede}, this implies
\begin{equation} \label{boundvv}
|\vv(1, \rho_n)| \ge c > 0, \quad n \ge 1.
\end{equation}
The estimate for $\dot r(\rho_n^2)$ follows from \cite[Lemma~5.4]{Hryn11}:
\begin{equation} \label{boundr}
|\dot r(\rho_n^2)| \le C n^{-2}, \quad n \ge 1.
\end{equation}
Combining \eqref{findMn}, \eqref{boundvv}, and \eqref{boundr}, we arrive at $M_n \ge \de_0$ for $n \ge 1$.

Following the proof of the sufficiency in Theorem~\ref{thm:nsc}, we obtain the asymptotics \eqref{asymptvvn} and \eqref{asymptrn} with the uniformly bounded remainder sequences:
$$
\| \{ s_n \}_1^{\infty} \|_{l_2} \le C, \quad \| \{ g_n \}_1^{\infty} \|_{l_2} \le C.
$$
Taking \eqref{findMn} into account, we arrive at the asymptotics \eqref{asymptM} with the estimate $\| \{ \eta_n \}_1^{\infty} \|_{l_2} \le \Omega_0$.

Thus, for every pairs $\bigl( d(\rho), \{ \rho_n \}_1^{\infty} \bigr)$ and $\bigl( \tilde d(\rho), \{ \tilde \rho_n \}_1^{\infty}\bigr)$ in $\mathcal S_{\Omega,\de}$, the corresponding data $\{ \rho_n, M_n \}_1^{\infty}$ and $\{ \tilde \rho_n, \tilde M_n \}_1^{\infty}$ lie in $\mathcal S^0_{\Omega_0, \de_0}$. Therefore, Proposition~\ref{prop:uni} implies the estimate \eqref{estqdir}, where the constant $C$ depends on the fixed parameters of Theorem~\ref{thm:uni}. Following the proof of Theorem~\ref{thm:loc} step by step for $\bigl( d(\rho), \{ \rho_n \}_1^{\infty} \bigr)$ and $\left( \tilde d(\rho), \{ \tilde \rho_n \}_1^{\infty}\right)$ in $\mathcal S_{\Omega,\de}$, we obtain the uniform estimate \eqref{difMn}. Combining \eqref{estqdir} and \eqref{difMn}, we arrive at the desired uniform estimate \eqref{estq}.
\end{proof}

\section{Quasiperiodic eigenvalues} \label{sec:eigen}

In this section, we reformulate the main results by using the eigenvalues of the quasiperiodic problem $\mathcal L$ instead of the Hill-type discriminant $d(\rho)$. For this purpose, we use the recovery of an entire function from its zeros as an infinite product and the uniform stability of this recovery (see \cite{But22}).

The problem $\mathcal L$ is self-adjoint and has a countable set of real eigenvalues, which coincide with 
with the squared zeros of the function $\Delta(\rho)$ defined by \eqref{defDelta}.
According to \eqref{intd}, there holds
\begin{equation} \label{intDelta}
\Delta(\rho) = 2 a_+ \cos \rho - 2 + \frac{b \sin \rho}{\rho} + \frac{\delta(\rho)}{\rho}, 
\end{equation}
where $a_+ \in \mathbb R \setminus (-1,1)$, $b \in \mathbb R$, and $\de \in PW_-(0,1)$. Consequently, the eigenvalues $\{ \nu_0^2\} \cup \{ (\nu_n^{\pm})^2\}_1^{\infty}$ of $\mathcal L$ (counting with multiplicities) possess the asymptotics (see \cite{Plak78}):
\begin{equation} \label{asymptnu}
\nu_n^{\pm} = 2\pi n \pm \theta + \frac{b}{4\pi a_+ n} + \frac{\varkappa_n^{\pm}}{n}, \quad \{ \varkappa_n^{\pm}\} \in l_2, \quad n \ge 1,
\end{equation}
where $\theta = \arccos a_+^{-1} \in (0,\pi)$.

We call the problem $\mathcal L$ positive definite if all its eigenvalues are positive. This condition is non-restrictive, since it can be achieved by a shift of the spectrum $\rho^2 \mapsto \rho^2 + c$. Then, without loss of generality, we may assume that $\mathcal V := \{ \nu_0 \} \cup \{ \nu_n^{\pm}\}_1^{\infty} \subset \mathbb R_+$. Below we sometimes denote $\nu_0^+ := \nu_0$.

The standard technique based on Hadamard's Factorization Theorem (see, e.g., \cite[Theorem~1.1.4]{FY01}) allows us to reconstruct the Hill-type discriminant $d(\rho)$ from the values $\mathcal V$:
\begin{equation} \label{prodDelta}
\Delta(\rho) = 2 (a_+ - 1) \prod_{n = 0}^{\infty} \frac{(\nu_n^+)^2 - \rho^2}{(2\pi n + \theta)^2} \prod_{n = 1}^{\infty} \frac{(\nu_n^-)^2 - \rho^2}{(2\pi n - \theta)^2}, \quad d(\rho) = \Delta(\rho) + 2.
\end{equation} 

The infinite products in \eqref{prodDelta} converge absolutely and uniformly on compact sets in the $\rho$-plane for every sequence $\mathcal V$ satisfying \eqref{asymptnu}.
Obviously, the parameters $\theta$ and $a_+$ in \eqref{prodDelta} can be found from the asymptotics \eqref{asymptnu}:
\begin{equation} \label{findthe}
\theta = \lim_{n \to \infty} (\nu_n^+ - 2 \pi n), \quad a_+ = (\cos \theta)^{-1}.
\end{equation}

Denote by $\mathcal D_+(a_+, b)$ ($a_+ \in \mathbb R \setminus (-1,1)$, $b \in \mathbb R$) the class of functions $d(\rho)$ satisfying \eqref{intd} such that all the zeros of $d(\sqrt \la) - 2$ are positive. The relation \eqref{prodDelta} provides the bijection between the class $\mathcal D_+(a_+, b)$ of functions $d(\rho)$ and the sequences $\mathcal V \subset \mathbb R_+$ with the asymptotics \eqref{asymptnu}. Moreover, \cite[Theorem~7]{But22} implies the uniform stability of recovering the function $d(\rho)$ from $\mathcal V$:

\begin{prop}[\hspace*{-3pt}\cite{But22}] \label{prop:But}
Let $a_+ \in \mathbb R \setminus (-1,1)$, $b \in \mathbb R$, and $\Omega > 0$ be fixed. Then, for any two sequences $\mathcal V = \{ \nu_0\} \cup \{ \nu_n^{\pm}\}_1^{\infty}$ and $\tilde{\mathcal V}	= \{ \tilde \nu_0\} \cup \{ \tilde \nu_n^{\pm}\}_1^{\infty}$ in $\mathbb R_+$ satisfying the asymptotics \eqref{asymptnu} and the corresponding estimates
\begin{equation} \label{estV}
\left( \nu_0^2 + \sum_{n = 1}^{\infty} n^2 \bigl( (\varkappa_n^+)^2 + (\varkappa_n^-)^2\bigr) \right)^{1/2} \le \Omega,
\end{equation}
there holds
\begin{equation} \label{uniD}
\| \rho(d(\rho) - \tilde d(\rho)) \|_{L_2(\mathbb R)} \le C \| \mathcal V - \tilde{\mathcal V} \|_{l_2^1},
\end{equation}
where the functions $d(\rho)$ and $\tilde d(\rho)$ are constructed by \eqref{prodDelta} and
\begin{equation} \label{proddt}
\tilde d(\rho) = 2 (a_+ - 1) \prod_{n = 0}^{\infty} \frac{(\tilde \nu_n^+)^2 - \rho^2}{(2\pi n + \theta)^2} \prod_{n = 1}^{\infty} \frac{(\tilde \nu_n^-)^2 - \rho^2}{(2\pi n - \theta)^2} + 2,
\end{equation}
respectively, 
the constant $C$ depends only on $a_+$, $b$, and $\Omega$, and
\begin{equation} \label{normdifV}
\| \mathcal V - \tilde{\mathcal V} \|_{l_2^1} := \left( (\nu_0 - \tilde \nu_0)^2 + \sum_{n = 1}^{\infty} n^2 \bigl( (\nu_n^+ - \tilde \nu_n^+)^2 + (\nu_n^- - \tilde \nu_n^-)^2\bigr)\right)^{1/2}.
\end{equation}
\end{prop}

Let us apply the above arguments to the following inverse spectral problem:

\begin{ip} \label{ip:eigen}
Given $\mathcal V$, $\{ \rho_n\}_1^{\infty}$, and $\{ \sigma_n \}_1^{\infty}$, find $q$, $a$, and $h$.
\end{ip}

By using \eqref{prodDelta} and \eqref{findthe}, Inverse Problem~\ref{ip:eigen} is equivalently reduced to Inverse Problem~\ref{ip:1}, which implies the uniqueness of solution.
The necessary and sufficient conditions of solvability, the local solvability and stability, and the uniform stability of Inverse Problem~\ref{ip:eigen} are obtained as corollaries of Theorems~\ref{thm:nsc}, \ref{thm:loc}, and~\ref{thm:uni}, respectively:

\begin{cor}
Let $a \in \mathbb R \setminus \{-1, 0, 1\}$ be fixed. Then, for a collection $\mathcal V$, $\{ \rho_n\}_1^{\infty}$, $\{ \sigma_n\}_1^{\infty}$ to be the spectral data of a positive definite problem $\mathcal L$ of form \eqref{eqv}--\eqref{bc}, it is necessary and sufficient to fulfill the requirement $\mathcal V \subset \mathbb R_+$, the asymptotics \eqref{asymptnu}, and conditions 2--5 of Theorem~\ref{thm:nsc}, where the function $d(\rho)$ is defined by \eqref{prodDelta}.
\end{cor}

\begin{cor} \label{cor:loc}
Let $q \in L_{2,\mathbb R}(0,1)$, $a \in \mathbb R \setminus \{ -1, 0, 1 \}$, and $h \in \mathbb R$ be fixed, and let $\mathcal V \subset \mathbb R_+$, $\{ \rho_n \}_1^{\infty}$, $\{ \sigma_n \}_1^{\infty}$ be the spectral data of $\mathcal L(q, a, h)$ in the sense of Inverse Problem~\ref{ip:eigen}. There exists $\eps > 0$ (which depends on $q$, $a$, and $h$) such that, if reals $\tilde{\mathcal V} \subset \mathbb R_+$ and $\{ \tilde \rho_n^2 \}_1^{\infty}$ satisfy the inequalities
\begin{equation} \label{ineqV}
\| \mathcal V - \tilde{\mathcal V} \|_{l_2^1} \le \eps, \quad 
\| \{ \rho_n - \tilde \rho_n \}_1^{\infty} \|_{l_2^1} \le \eps
\end{equation}
and the condition \eqref{condtd}, then there exists a unique $\tilde q \in L_{2,\mathbb R}(0,1)$ such that $\tilde{\mathcal V}$, $\{ \tilde \rho_n \}_1^{\infty}$, and $\{ \sigma_n \}_1^{\infty}$ are the spectral data of $\mathcal L(\tilde q, a, h)$ in the sense of Inverse Problem~\ref{ip:eigen}. Moreover, the following estimate holds:
\begin{equation} \label{unieig}
\| q - \tilde q \|_{L_2} \le C \left( \| \mathcal V - \tilde{\mathcal V} \|_{l_2^1} + \| \{ \rho_n - \tilde \rho_n \}_1^{\infty} \|_{l_2^1} \right),
\end{equation}
where the constant $C > 0$ depends on $q$, $a$, and $h$. Above, the norm $\| \mathcal V - \tilde{\mathcal V} \|_{l_2^1}$ is defined in \eqref{normdifV} and the function $\tilde d(\rho)$ in \eqref{condtd} is defined by the formula \eqref{proddt}.
\end{cor}

\begin{proof}
The inequality \eqref{ineqV} for any $\eps < \infty$ implies the asymptotics \eqref{asymptnu} for the sequence $\tilde{\mathcal V}$. Therefore, the function $\tilde d(\rho)$ is correctly defined by formula \eqref{proddt} and admits the representation \eqref{intdt} with the same $a_+$ and $b$ as $d(\rho)$ and with some $\tilde D \in L_{2,\mathbb R}(0,1)$. Furthermore, by virtue of Proposition~\ref{prop:But}, there holds the estimate \eqref{uniD}, which allows us to immediately obtain Corollary~\ref{cor:loc} from Theorem~\ref{thm:loc}.
\end{proof}

Similarly to $\mathcal S_{\Omega,\de}$, define the set $\mathscr S_{\Omega,\de}$ of the pairs $\left( \mathcal V, \{ \rho_n \}_1^{\infty}\right)$ that satisfy all the requirements of Definition~\ref{def:S} by only replacing ``$d(\rho)$ has the form \eqref{intd} with $b = ah+2a_+\om$ and $\| D \|_{L_2} \le \Omega$'' with ``$\mathcal V \subset \mathbb R_+$, \eqref{asymptnu} and \eqref{estV} hold''. Therein, the function $d(\rho)$ is supposed to be constructed by \eqref{prodDelta}.
Note that, by Proposition~\ref{prop:But}, the bound \eqref{estV} directly implies $\| D \|_{L_2} \le \Omega$. Consequently, Theorem~\ref{thm:uni} yields the following corollary on the uniform stability of Inverse Problem~\ref{ip:eigen}:

\begin{cor}
Let $a$, $h$, $\om$, $\{ \sigma_n \}_1^{\infty}$, $\Omega$, and $\de$ be fixed values satisfying the hypothesis of Theorem~\ref{thm:uni}. Then, for any two pairs $\left( \mathcal V, \{ \rho_n \}_1^{\infty}\right)$ and $\left( \tilde{\mathcal V}, \{ \tilde \rho_n \}_1^{\infty}\right)$ in $\mathscr S_{\Omega,\de}$, the corresponding potentials $q$ and $\tilde q$ recovered from the spectral data $\left( \mathcal V, \{ \rho_n \}_1^{\infty}, \{ \sigma_n \}_1^{\infty} \right)$ and $\left( \tilde{\mathcal V}, \{ \tilde \rho_n \}_1^{\infty}, \{ \sigma_n \}_1^{\infty} \right)$, respectively, satisfy the estimate \eqref{unieig}, where the constant $C$ depends only on $a$, $h$, $\om$, $\{ \sigma_n \}_1^{\infty}$, $\Omega$, and $\de$.
\end{cor}

\medskip

{\bf Funding.} This work was supported by Grant 24-71-10003 of the Russian Science Foundation, https://rscf.ru/en/project/24-71-10003/.

\medskip

{\bf Competing Interests.} This paper has no competing interests.

\end{document}